\documentclass[a4paper,11pt,english]{article}

\usepackage{babel}
\usepackage{amsmath,amsfonts,amssymb,amsthm}
\usepackage{latexsym}
\usepackage{mathtools}
\usepackage[utf8]{inputenc}
\usepackage[dvipsnames,usenames]{xcolor}
\usepackage{fancyhdr}
\usepackage[nottoc]{tocbibind}
\usepackage{indentfirst}
\usepackage{paralist}
\usepackage{hyperref}
\usepackage{multirow}
\hypersetup{
colorlinks=true,
linkcolor=blue,
filecolor=magenta,
urlcolor=blue,
}

\usepackage{tikz-cd}
\usepackage{enumerate}
\usepackage{makeidx}
\usepackage{multicol}
\usepackage{comment}
\usepackage[all,cmtip]{xy}
\usepackage{appendix}

\usepackage{lipsum}

\newcommand\blfootnote[1]{%
  \begingroup
  \renewcommand\thefootnote{}\footnote{#1}%
  \addtocounter{footnote}{-1}%
  \endgroup
}

\newtheorem{theorem}{Theorem}[section]
\newtheorem{proposition}[theorem]{Proposition}
\newtheorem{lemma}[theorem]{Lemma}
\newtheorem{corollary}[theorem]{Corollary}

\theoremstyle{definition}
\newtheorem{definition}[theorem]{Definition}

\newtheorem{remark}[theorem]{Remark}
\newtheorem{notation}[theorem]{Notation}
\newtheorem{example}[theorem]{Example}

\numberwithin{equation}{section}

\newcommand{\Q}{\mathbb Q}

\newcommand{\GL}{\text{GL}}

\DeclareMathOperator{\Gal}{Gal}
\DeclareMathOperator{\Aut}{Aut}

\DeclareMathOperator{\PGL}{{\rm PGL}}

\DeclareMathOperator{\Char}{char}

\newcommand{\nr}{{\scriptstyle\rm nr}}

\newcommand{\cA}{\mathcal{A}}
\newcommand{\cB}{\mathcal{B}}
\newcommand{\cC}{\mathcal{C}}
\renewcommand{\bar}{\overline}

\title{On the conductor of Ciani plane quartics}

\usepackage{authblk}

 \author[1]{Irene Bouw}
 \author[2]{Nirvana Coppola}
 \author[3,4]{Elisa Lorenzo Garc\'ia}
 \author[4]{Anna Somoza}
 \affil[1]{Ulm University, Germany}
 \affil[2]{Vrije Universiteit Amsterdam, The Netherlands}
 \affil[3]{Univerisité de Neuchâtel, Switzerland}
 \affil[4]{Université de Rennes 1, France}

\begin{document}
	\maketitle
	\begin{abstract}
In this paper we determine the conductor exponent of non-special Ciani quartics at primes of potentially good reduction in terms of the Ciani invariants. As an intermediate step in order to do so, we provide a reconstruction algorithm to construct Ciani quartics with given invariants. We also discuss how to descend the provided model to be defined over the same field as the invariants. 
\end{abstract}

	\section{Introduction}\label{sec:intro}

\blfootnote{\textit{2020 Mathematics Subject Classification}. 11G20, 14G20, 14H10, 14H20, 14H45, 14H50, 14Q05.}
\blfootnote{\textit{Keywords}. Plane quartic curves, Ciani quartic curves, invariants, minimal discriminant, stable reduction, conductor.}
Let $(K, \nu)$ be a complete local field of characteristic zero with valuation $\nu$, whose residue field is an algebraically closed field $k$ of odd characteristic $p>2$. We start by recalling some facts on elliptic curves. For $j\in K$ there exists an elliptic curve $E/K$ with $j(E)=j. $ It has potentially good reduction if and only if $\nu(j(E))\geq 0$. Moreover, if the valuation of $j\in K$ is non-negative, there exists an elliptic curve $E_0/K$ with $j(E_0)=j$ that has good reduction over~$K$.  A motivation for the results in this paper is to explore whether  similar statements hold for  non-hyperelliptic curves of genus $3$, i.e.~plane quartic curves.

{The Dixmier--Ohno invariants $DO(Y)\in \bar{K}^{14}$, where $\bar{K}$ is an algebraic closure of $K$, are a set of invariants for   plane quartic curves $Y/\bar{K}$   that determine $\overline{K}$-isomorphism classes. One of these is the discriminant $\Delta(Y)$. 
In principle, one should be able to read off from the Dixmier--Ohno invariants  all information of $Y$ that only   depends on the $\overline{K}$-isomorphism class. It is known for example how to read off the  automorphism group $\Aut_{\overline{K}}(Y)$ from the Dixmier--Ohno invariants, but for other information it is less clear how to do this in practice.  

A first difference between elliptic curves and plane quartics is that the field of moduli need not be a field of definition. However, given a smooth quartic $Y/\overline{K}$ with  $|\Aut_{\overline{K}}(Y)|>2$ and a set $DO(Y)\in K^{14}$ of Dixmier--Ohno invariants, there exists a $K$-model of $Y$ with those (projective) invariants, see e.g.~\cite{LRRS}. }

In this paper, we restrict to the locus of Ciani quartics. A \emph{Ciani quartic} is a smooth quartic $Y$ whose automorphism group contains a subgroup $V\simeq C_2^2$ with $g(Y/V)=0$. We call a subgroup $V$ with these properties a \emph{Ciani subgroup}. Ciani quartics form a $3$-dimensional stratum in the moduli space of plane quartics. In fact, it is the largest-dimensional stratum in the moduli space of plane quartics with {$|\Aut_{\overline{K}}(Y)|>2$}, where we consider the stratification by automorphism group.  In our previous paper \cite{BCKKLS} we defined a set of invariants $\underline{I}=\underline{I}(Y)=(I_3,I_3', I_3'', I_6) \in K^4$ for Ciani quartics $(Y, V)$, where $V$ is a Ciani subgroup. The main result of \cite{BCKKLS} is a recipe to determine the type of the stable reduction from the Ciani invariants.  

This set of Ciani invariants is much smaller than the Dixmier--Ohno invariants, and therefore easier to work with in practice. It is convenient to consider the Ciani invariants as a point $[\underline{I}]\in \mathbb{P}_{1,1,1,2}(K)$ in a weighted projective space since this determines the $\bar{K}$-isomorphism class. We usually assume that all invariants have non-negative valuation and that at least one has valuation zero, i.e.~the invariants are \emph{normalised}. The discriminant of a Ciani quartic can be expressed in terms of the Ciani  invariants. For a set of Ciani invariants $\underline{I}$, we write $\Delta(\underline{I})=2^{20}I_3(I_3'')^4I_6^2$. It is the discriminant of a curve with (exactly) these invariants. If $\underline{I}$ is normalised, it can be considered as the minimal discriminant of curves in the corresponding $\overline{K}$-isomorphism class, i.e. having the minimal valuation among all the integral $\bar{K}$-models.

{In the present paper, we focus on a more arithmetic question, namely we study the conductor exponent of a Ciani quartic. Choose $\underline{I}\in K^4$ a normalised set of Ciani invariants with $\Delta(\underline{I})\neq 0$. 
Let $Y/K$ be a Ciani quartic with invariants $[\underline{I}]$. For Ciani quartics, the conductor exponent $f_p(Y)$ of $Y$ is zero if and only if  $Y$ has good reduction to characteristic $p$ (Corollary \ref{cor:conductor_good}). 
There may exist more than one non-($K$-)isomorphic $K$-model of $Y$, and the conductor exponent depends on the chosen  model, in general. One of our main results characterises, given a set of invariants $\underline{I}$, whether there exists a $K$-model  of $Y$ with good reduction, i.e.~with $f_p=0$, under the assumption that $\Aut_{\overline{K}}(Y)=V$. The last condition is equivalent to $\Aut_{\overline{K}}(Y)$ containing a unique Ciani subgroup. If this condition is satisfied, we say that $Y$  is \emph{non-special}. There are two complementary cases: if $Y$ has potentially good reduction to characteristic $p$, the reduction $\overline{Y}$ is either a smooth quartic (good quartic reduction) or hyperelliptic (good hyperelliptic reduction). In our proofs, we treat both cases separately.}

Our results are stated in terms of some extra invariants $Q$ and $R$, which are defined in Sections \ref{sec:recons} and \ref{sec:hyper_cond}. We refer to Section \ref{sec:recons} for a definition of $Q$ in terms of the Ciani invariants. It satisfies that $Y$ is non-special if and only if $Q\neq 0$, see Section \ref{sec:special}. It also occurs as the discriminant of a polynomial $\mathcal{P}$ of degree $3$, introduced in Section \ref{sec:recons}, which is important in the description of a field extension of $K$ over which we can find a model with stable reduction. The invariant $R$ is defined in Equation \ref{def:invR}.  It is also related to the extension of $K$ over which $Y$ acquires stable reduction. We can summarise our main result as follows. Here we state it for simplicity in the case where $K=\Q_p^{\nr}$ is the maximal unramified extension of $\Q_p$.

\begin{theorem}\label{thm:main}
 Let $p\neq 2,3$ and $K=\Q_p^{\nr}$. Let $\underline{I}\in K^4$ be a normalised set of invariants. Let $Y/\overline{K}$ be a Ciani quartic  with $\underline{I}(Y)=\underline{I}$. Assume that $Y$ is smooth and that $\Aut_{\overline{K}}(Y)=V$. 
\begin{itemize}
    \item[(I)] Assume that $p\nmid \Delta(\underline{I}).$
    \begin{itemize}
        \item[(a)] The curve $Y/\overline{K}$ has good quartic reduction.
        \item[(b)]  There exists a $K$-model  of $Y$ with good reduction if and only if $\nu(Q)=0.$
    \end{itemize}
\item[(II)] Assume that $\nu(I_3)=0, \nu(I_3')\geq e, \nu(I_3'')=2e, \nu(I_6)=3e$ for some $e>0.$
\begin{itemize}
    \item[(a)] The curve $Y/\overline{K}$ has good hyperelliptic reduction.
    \item[(b)] There exists a $K$-model of $Y$ with good reduction if and only if $e$ is even and the polynomial $\mathcal{P}$ splits completely over $K$. This is equivalent to $e$ being even, $\nu(Q)\equiv0,2,4\pmod{6}$ and $3\nu(R)>\nu(Q)$ if $\nu(Q)\not\equiv0\pmod{6}$.
\end{itemize}
\item[(III)] In all other cases, the curve $Y/\overline{K}$ has bad reduction.
\end{itemize}
 
\end{theorem}

Our results are more precise. For the precise statements, we refer to Propositions \ref{prop:smoothmodel_new} and \ref{prop:val_equal} for (I),  to Proposition \ref{prop:hyperconductor} for (II), and to  Lemmas \ref{lem:badprimes} and \ref{lem:potgoodhyper} for (III). If the Ciani curve has potentially good reduction, i.e.~in the situation of (I) or (II),  we determine the minimal value for the conductor exponent $f_p$ among all $K$-models of $Y$. We also find the concrete model, together with the minimal field over which it has good reduction. In case (III), we prove that the conductor, for any $K$-model, will be always positive, see Corollary \ref{cor:conductor_good}. 

Theorem \ref{thm:main} points to a  further difference between elliptic curves and Ciani quartics (or plane quartics in general). Assume that the invariants $\underline{I}$ are normalised. In the case that $\nu(\Delta(\underline{I}))=0$ there need not exist a curve $Y/K$ with good reduction over the minimal field $K$ with those invariants. In fact, in the case of good quartic reduction,  we show that such a $K$-model exists if and only if the automorphism group  of the reduction $\overline{Y}$ is equal to $\Aut_{\overline{K}}(Y)$, which we assumed to be $V$.  In the case of good hyperelliptic reduction this is not true: the automorphism group of $\overline{Y}$ is always strictly larger than $V$, but it is possible for the conductor exponent to be $0$. However, if there is no $K$-model of $Y$ with good reduction over $K$, i.e.~with $f_p(Y)=0$, then the automorphism group of $\overline{Y}$ is strictly larger than the group generated by the hyperelliptic involution and the elements of the fixed Klein $4$-group $V$, see Remark \ref{rem:autospecial}.  

{The paper is structured as follows. In Section \ref{Sec:prop} we introduce Ciani plane quartics, their different models and invariants. In Section \ref{sec:reconst} we give a reconstruction algorithm to obtain a Ciani plane quartic with given Ciani invariants. We characterise when they are special and we compute their twists in the non-special case. In Section \ref{sec:stablered} we recall the basic definitions of stable reduction and the  conductor of a curve and we present the main results that will allow us to compute the conductor of a Ciani plane quartic. Our main results are stated and proved in Section \ref{sec:badred}. Finally, in Section \ref{Sec:bounds} we discuss how to use our results to bound the conductor of a Ciani plane quartic.}

\subsection*{Notation}
Throughout the paper, we use the following notation.
\begin{table}[h]
\small{    \begin{tabular}{rl}
$K$ & a complete local field, usually assumed to be $\mathbb{Q}_p^{\nr}$, with valuation $\nu$, \\
$Y_0$ & a $K$-model of a Ciani quartic as in Definition \ref{def:standard_model},\\
$Y_1$ & a standard model as in Definition \ref{def:standard_model}, \\
$\underline{I}=(I_3,I'_3,I''_3,I_6)\in K^4$ & the set of Ciani invariants as in Equation \ref{eq:inv},\\ $\Delta(\underline{I})=2^{20}I_3I''^4_3I^2_6$ & the discriminant as in Equation \ref{eq:disc},\\ 
$\underline{I}_\nu$ & the normalisation of $\underline{I}$ at $\nu$, Definition \ref{def:normalise},\\
$[\underline{I}]\in \mathbb{P}^3_{1,1,1,2}$ & the point corresponding to $\underline{I}\in K^4$ in weighted projective space,\\
$Y(\underline{I})/\bar{K}$ & a Ciani quartic with projective Ciani invariants $[\underline{I}]\in\mathbb{P}^3_{1,1,1,2}$,\\
$L/K$ & the field of decomposition of the polynomial $\mathcal{P}$ defined in Equation (\ref{eq:coeffP}),\\
$Y_0(\underline{I})/K$ & a Ciani quartic with Ciani invariants $[\underline{I}]$ as in Proposition~\ref{prop:reconsK},\\
$M/K$ & the unique minimal extension over which $Y(\underline{I})$ has stable reduction.
    \end{tabular}}
    \caption{Notation}
    \label{tab:notation}
\end{table}

{\subsection*{Acknowledgements}The research of the second author is supported by the NWO Vidi grant No. 639.032.613, New Diophantine Directions.  The research of the third author is partially funded by the Melodia ANR-20-CE40-0013 project. }

\section{Properties of Ciani quartics}\label{Sec:prop}
In this section, we recall some general properties of Ciani plane quartic curves (usually abbreviated as Ciani quartics), and introduce some useful notation.

\subsection{The standard model of a Ciani quartic}\label{sec:standard}

 Let $K$ be, as in the previous section.

\begin{definition}
	A \emph{Ciani plane quartic curve}, or Ciani quartic, is a smooth non-hyperelliptic curve $Y/K$ of genus $g=3$ such that there exists \(V\subset\Aut_{\overline{K}}(Y)\), with $V$ isomorphic to $C_2^2$ and \(g(Y/V) = 0\). A subgroup $V$ satisfying these conditions is called a \emph{Ciani subgroup}. 
\end{definition}

 The Ciani subgroup is part of the data of a Ciani quartic. If the Ciani subgroup is clear from the context, we will sometimes omit it from the notation. 
 We say that a Ciani quartic $(Y,V)$ is defined over $K$ if and only if $Y$ is defined over $K$ and $\,^\sigma V=V$ for all $\sigma\in\Gal(\overline{K}/K)$. 

Let $(Y, V)$ be a Ciani quartic. The map $\phi:Y\to Y/V=:X$ is branched above $6$ points. Let $\sigma\in V\setminus \{e\}$ one of the nontrivial elements. Then $\sigma$  generates the inertia above exactly $2$ of the branch points of $\phi$, and the genus of $Y/\langle\sigma\rangle$ is $1$.  

\begin{definition}\label{def:special}
    A Ciani quartic $Y$ is called \textit{non-special} if there is a unique Ciani subgroup. Otherwise, we call it \textit{special}. 
\end{definition}

The following lemma follows immediately from the classification of possible automorphism groups of smooth quartics, see for example \cite[Section 3.1]{LRRS} or \cite{Vermeulen}
for this. 

\begin{lemma}
A Ciani quartic is non-special if and only if  $\Aut_{\overline{K}}(Y)\simeq C_2^2$. In particular, every non-special Ciani quartic $(Y,V)$ is defined over a field $K$ if and only if the curve $Y$ is defined over $K$.
\end{lemma}

\begin{figure}\label{eq:Non-HyperellipticStrata}
\begin{equation*}
\xymatrix{
    & &  & C_2^2 \ar@{-}[d] & & &  \dim = 3\\
  & & & D_4 \ar@{-}[ld] \ar@{-}[rd] & & &  \dim = 2 \\
  & &    G_{16} \ar@{-}[ld] \ar@{-}[rd] & & S_4 \ar@{-}[ld] \ar@{-}[rd]  & & \dim = 1\\
  & G_{48} & & G_{96} & & G_{168} &  \dim = 0
}
\end{equation*} 
\caption{Stratification by automorphisms of Ciani quartics}
\label{fig:stratification}
\end{figure}
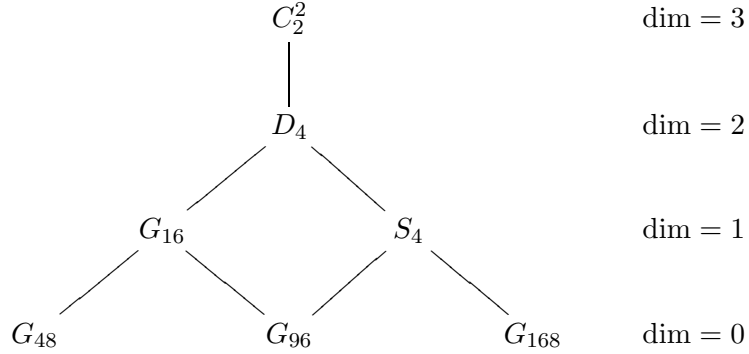

\begin{example}\label{exa:fieldD4}
We consider the Ciani quartic 
\[
Y_{a,b}/K:\; ax^4+y^4+z^4+bx^2y^2+xyz^2=0,
\]
with $a,b\in K$ such that its discriminant $\Delta(Y_{a,b})=2^{20}
(a-2)^2(a+2)^2(4a-b^2 -8)^4(4a-b^2 +8)^4\neq 0$.
The automorphism group of $Y$ contains the dihedral group $D_4$ of order $8$ as a subgroup generated by
\[
\left\langle \sigma_1=\begin{pmatrix}i&0&0\\0&-i&0\\0&0&1\end{pmatrix}, \,\, \sigma_2=\begin{pmatrix} 0&1/\sqrt[4]{a}&0\\\sqrt[4]{a}&0&0\\0&0&1\end{pmatrix}\right\rangle.
\]
and hence contains at least 2 Ciani subgroups: $V=\langle\sigma_1^2,\sigma_2 \rangle$ and $V'=\langle\sigma_1^2,\sigma_1\sigma_2 \rangle$. If $a\notin (K^*)^2$ then $V$ and $V'$ are (Galois-)conjugated subgroups over the quadratic extension $K(\sqrt{a})$ but not conjugated in  $\Aut_{\overline{K}}(Y_{a,b})\subset \PGL_3(\overline{K})$. Hence $Y_{a,b}$ is  defined over $K$ as a curve, but  the Ciani quartic $(Y_{a,b},V)$ is not.
\end{example}

\begin{remark} All special Ciani quartics belong to the family $Y_{a,b}$ in Example \ref{exa:fieldD4}, see \cite[Thm. 3.3]{LRRS}.
\end{remark}

\begin{definition}\label{def:standard_model}
Let $(Y, V)$ be a Ciani quartic over the algebraic closure $\overline{K}$ of  $K$. 
\begin{itemize}
    \item[(a)] A $K$-\emph{model} of $Y$ is a smooth quartic $Y_0/K$ such that $Y_0\otimes_K \overline{K}\simeq Y$. 
        \item[(b)]   A \emph{standard $K$-model}  of $(Y, V)$ is 
    a $K$-model  $Y_1$ of $Y$ given by an equation
    \begin{equation}\label{eq:standard_model}
    Y_1:\; Ax^4+By^4+Cz^4+ay^2z^2+bx^2z^2+cx^2y^2=0,
    \end{equation}
    such that the  elements of $V$ act as $(x:y:z)\mapsto (\pm x:\pm y:z). $
\end{itemize} 
\end{definition}

\begin{notation}\label{not:sigma} Let $Y_1$ be a standard model of a Ciani quartic.  We write
\[\begin{split}
\sigma_a(x:y:z)&=(-x:y:z),\\
\sigma_b(x:y:z)&=(x:-y:z),\\
\sigma_c(x:y:z)&=(-x:-y:z)
\end{split}
\]
for the nontrivial elements of $V$. For $i\in \{a,b,c\}$, we write $E_i=Y_1/\langle \sigma_i\rangle.$
\end{notation}

It is well-known that every Ciani quartic over a field of characteristic not $2$ admits a standard model over a finite extension.
In the following lemma, we give a bound on the field extension needed.

\begin{lemma}\label{lem:fod_standard}
Let $(Y_0, V)$ be a Ciani quartic over $K$. Then there exists a Galois extension $L/K$ with $\Gal(L/K)<S_3$ such that $Y_0$ admits a standard model $Y_1$ over $L$ and such that $Y_0\otimes_K L\simeq _L Y_1$.
\end{lemma}

\begin{proof}
Since $(Y_0,V)$ is defined over $K$, the map $\varphi: Y_0 \rightarrow X_0:=Y_0/V$ is defined over $K$, as well. The curve $X_0$ is a conic. We denote the branch locus of $\varphi$ by $D$. The elements of $V$ of order $2$, which we call $\sigma_a,\sigma_b,\sigma_c$, are each branched at a pair of points on $X_0\otimes_K \overline{K}$.  It follows that $\Gamma_K=\Gal(\overline{K}/K)$ acts on the branch locus $D$ by permuting the three pairs of points. Let $L$ be the field extension such that each pair is rational over $L$. Note that $\Gal(L/K)$ is a subgroup of $S_3$.
Over $L$ we may write $D$ as the sum of three effective divisors $D_a, D_b, D_c$, where $D_i$ is the divisor corresponding to the branch points with inertia generator $\sigma_i$.

We claim that $Y_0$ admits a standard model over $L$.  We write  $\ell_a,\ell_b,\ell_c$ for the three lines on $X_0 \otimes_K L$ with $(\ell_i)_0 = D_i$ for $i \in \{a,b,c\}$. The three lines are defined over $L$, and  define an embedding 
\[
\psi_L: X\otimes_K L\hookrightarrow \mathbb{P}^2_L.
\]
The coordinates $u, v, w$ of $\mathbb{P}^2_L$ define the images of the lines $\ell_a, \ell_b, \ell_c$, respectively. We denote the image of $X_0$ under $\psi_L$ by $X_1$. It is a conic with equation
\begin{equation}\label{eq:conic}
Au^2+Bv^2+Cw^2+avw+buw+cuv=0.
\end{equation}

We define three functions $x,y,z$ on $Y_0$ satisfying  $x^2=u, y^2=v, z^2=w $. The functions $(x,y,z)$ define an embedding
\[
\varphi_L: Y_0 \otimes_K L\to \mathbb{P}^2_L
\]
such that the image  $Y_1$ of $Y_0 \otimes_K L$ is given by
\[
Ax^4+By^4+Cz^4+ay^2z^2+bx^2z^2+cx^2y^2=0.
\]
Hence $Y_1$ is a standard $L$-model.
\end{proof}

\begin{remark} In the proof of previous lemma, one may alternatively characterise  the coordinates $x,y,z$ and the embedding $\varphi_L$ by the property that the nontrivial elements of $V$ are the diagonal matrices
\begin{equation}\label{eq:sigma}
\sigma_a=\begin{pmatrix} -1&0&0\\0&1&0\\ 0&0&1\end{pmatrix}, \quad \sigma_b=\begin{pmatrix} 1&0&0\\0&-1&0\\0&0&1\end{pmatrix}, \quad \sigma_c=\begin{pmatrix} 1&0&0\\0&1&0\\0&0&-1\end{pmatrix}.
\end{equation}
\end{remark}

\subsection{Ciani invariants}
The canonical embedding of a Ciani quartic $Y$ as a plane quartic in $\mathbb{P}^2$ only depends on the choice of a basis of $H^0(Y_{\overline{K}}, \Omega)$. Therefore, isomorphisms between Ciani quartics are induced by elements of $\GL_3(\overline{K}).$   If $\Char(K)\neq 2,3,5$, we have that two Ciani quartics are $\overline{K}$-isomorphic as plane quartics if and only they have the same Dixmier--Ohno invariants (see \cite[Thm. 4.1]{LLLR} for the cases of small characteristic).

In \cite[Section 3.1]{BCKKLS} we introduced a set of invariants for Ciani quartics $(Y, V)$ for $\Char(K)\neq 2$. 
For a standard model \eqref{eq:standard_model} $Y_1$ over a field $K$, the \emph{Ciani invariants} are defined by 
\begin{equation} \label{eq:inv}
\begin{split}
I_3 & = ABC \\
I'_3 & = A(a^2-4BC)+B(b^2-4AC)+C(c^2-4AB)\\
I''_3 & =-4ABC+Aa^2+Bb^2+Cc^2-abc\\
I_6 & =(a^2-4BC)(b^2-4AC)(c^2-4AB).
\end{split}
\end{equation}
These invariants are algebraically independent and generate the ring of invariants of the locus parametrising Ciani quartics in the moduli space of non-hyperelliptic curves of genus $3$. The index of the invariants indicates the degree, as used in Lemma \ref{trans}.
We write $\underline{I}(Y_1):=(I_3, I'_3, I''_3, I_6)\in K^4$  for the tuple of Ciani invariants of the standard model $Y_1/K$.
This tuple defines a point $[\underline{I}]$ in the weighted projective space $ \mathbb{P}^3_{1,1,1,2}$.
Two Ciani quartics are $\overline{K}$-isomorphic as Ciani quartics if and only their invariants define the same point in $\mathbb{P}^3_{1,1,1,2}(\overline{K})$. 

The discriminant of a standard model $Y_1$ (\ref{eq:standard_model}) is given by the formula:  
\begin{equation}\label{eq:disc}
\Delta(Y_1)=2^{20}I_3I''^4_3I^2_6.
\end{equation}
Note that this is an expression in terms of the invariants \eqref{eq:inv}: 
we sometimes denote it by $\Delta(\underline{I})$. 
The model $Y_1$ is smooth if and only if $\Delta(Y_1)\neq0$, see \cite{gelfand}.

In \cite{Coppola} one can find an expression for the Dixmier--Ohno invariants of a Ciani quartic in terms of the Ciani invariants. Since two non-isomorphic Ciani quartics $(Y, V_i)$ with the same underlying curve have different Ciani invariants, but the same Dixmier--Ohno invariants, it is not possible to express the Ciani invariants in terms of the Dixmier--Ohno invariants in general. 
We give a concrete example.

\begin{example}\label{exa:DOversusCiani}
We consider the Ciani quartic 
\[
Y_{r,s}:\; x^4+y^4+z^4+rz^2(x^2+y^2)+sx^2y^2=0,
\]
with $r,s\in K$ such that $\Delta(Y_{r,s})=2^{20}(s-2)^{10}(s+2)^6(-r^2+s+2)^4(r^2-4)^{12}\neq~0.$
The automorphism group of $Y$ contains the dihedral group $D_4$ of order $8$ as a subgroup, and hence contains 2 Ciani subgroups: the standard subgroup $V$ from \eqref{eq:sigma} and a second subgroup with generators
\[
V'=\left\langle \sigma_c=\begin{pmatrix}1&0&0\\0&1&0\\0&0&-1\end{pmatrix}, \qquad \sigma_d=\begin{pmatrix} 0&1&0\\1&0&0\\0&0&1\end{pmatrix}\right\rangle.
\]
Changing coordinates so that the matrix $\sigma_d$ is a diagonal matrix yields an isomorphic curve but now with $V'$ as standard subgroup:
\[
Y'_{r,s}:\; (2+s)x^4+(2+s)y^4+z^4+2rz^2(x^2+y^2)+(12-2s)x^2y^2=0.
\]
The curves $Y_{r,s}$ and $Y'_{r,s}$ have the same Dixmier--Ohno invariants, but the pairs $(Y_{r,s}, V)$ and $(Y'_{r,s}, V')$ do not have the same (projective) Ciani invariants. 
\end{example}

We introduce the following invariant that will be useful later on:
\begin{equation}\label{eq:Idef}
\begin{split}
I&=AB(a^2-4BC)(b^2-4AC)+BC(b^2-4AC)(c^2-4AB)\\&+CA(c^2-4AB)(a^2-4BC).
\end{split}
\end{equation}
The invariant $I$ can be expressed in terms of the other invariants via the relation
\begin{equation}\label{eq:Irel}
4I+ I_6 - {I'_3}^2 + 16I_3I''_3 + 2I'_3I''_3 - {I''_3}^2 = 0.
 \end{equation}

Let $(Y,V)$ be a Ciani quartic that is not necessarily in standard form. The following lemma allows to compute the Ciani invariants of $(Y, V)$. 

\begin{lemma}\label{trans} Let $F(x,y,z)=0$ be a plane quartic and $J$ a degree-$k$ invariant. Let $\lambda\in \overline{K}^{*}$ and $M\in \GL_3(\bar{K})$. Then
\begin{itemize}
\item[i)] $J(\lambda F)=\lambda^kJ(F)$,
\item[ii)] $J(\,^MF)=(\det(M))^{4k/3}\cdot J(F)$.
\end{itemize}
In particular, $\Delta(\lambda F)=\lambda^{27}\Delta(F)$ and $\Delta(\,^MF)=(\det(M))^{36}\Delta(F)$.
\end{lemma}

\begin{example}\label{exa:fieldD4C} The Ciani invariants of the Ciani quartic $(Y_{a,b},V)$ in Example \ref{exa:fieldD4} are:
$$
I_3=2^4a(b+2\sqrt{a})^2,\,\,I'_3=2^5a(48a-48\sqrt{a}b+2\sqrt{a}-4b^2+b),
$$
$$
I''_3=2^5a\sqrt{a}(8\sqrt{a}-4b+1),\,\,I_6=2^{14}a^2(2\sqrt{a}-b)(8\sqrt{a}+4b-1)^2.
$$

For the Ciani quartic $(Y_{a,b},V')$ we find the conjugated ones over $K(\sqrt{a})/K$. We notice here that $[\underline{I}(Y_{a,b},V)]\neq[\underline{I}(Y_{a,b},V')]$ for $(a,b)\neq(0,0)$ and $a\notin (K^*)^2$. 
\end{example}

Let $(Y,V)/\overline{K}$ be a Ciani quartic, and write $\underline{I}(Y,V)$ for its set of invariants. When the subgroup $V$ is clear from the context we may simply write $\underline{I}(Y)$. We also use the notation $Y(\underline{I})/\bar{K}$ to denote a Ciani quartic whose Ciani invariants are $[\underline{I}]$, i.e. $[\underline{I}(Y(\underline{I}))]=[\underline{I}]$.  In Proposition \ref{reconsquartic} we will explicitly see that such a Ciani quartic always exists.

\begin{proposition}\label{prop:exist}
Let $Y/\overline{K}$ be a Ciani quartic, and write $\underline{I} =\underline{I}(Y)$ for its set of invariants. Assume that $[\underline{I}]=(I_3 : I_3' : I_3'' : I_6)\in \mathbb{P}_{1,1,1,2}(K)$ has a representative over $K$. Then there exists a $K$-model $Y_0$ of $Y$. 
\end{proposition}

\begin{proof}
The statement follows from \cite[Prop. 2.1]{elg-twists}. See also \cite[Thm. 3.3]{LRRS}.
\end{proof}

\begin{remark} \label{rem:actiononV}In the setting of  Proposition \ref{prop:exist}, we have that  $(Y_0,V)\simeq_{\bar{K}}(Y_0,\,^\sigma V)$ for all $\sigma\in \Gal(\bar{K}/K)$. In particular, for each $\sigma\in \Gal(\bar{K}/K)$ there exists $M_\sigma\in\Aut_{\overline{K}}(Y_0)\simeq \PGL_3(\overline{K})$ such that $\,^\sigma V=M_\sigma^{-1}VM_\sigma$. If $Y_0$ is non-special, the subgroup $V\subset \Aut_{\overline{K}}(Y_0)$ is automatically $K$-rational. This is not always the case in the special case.  For instance, consider
\[
Y/\mathbb{Q}:\; x^4+y^4+xz^3=0
\]
and the Ciani subgroup $V$ generated by the automorphisms 
$$
\begin{pmatrix} 1&0&0\\0&-1&0\\0&0&1\end{pmatrix}, \qquad \begin{pmatrix} 1/\sqrt{3} & 0 & \zeta^2_3/\sqrt{3}\\
0 & 1 & 0\\
2\zeta_3/\sqrt{3} & 0 & -1/\sqrt{3}
\end{pmatrix},
$$
where $\zeta_3=e^{2\pi i/3}.$
The Ciani invariants of $(Y, V)$ are $(1:-24:-16:-256)\in\mathbb{P}^3_{1,1,1,2}(\mathbb{Q})$, but $V$ is not defined over $\mathbb{Q}$.
\end{remark}

\subsection{Normalisation of invariants}\label{sec:normalise}
The invariants of a Ciani quartic $Y/K$ form a tuple $\underline{I}(Y)\in K^4$. Sometimes it is very important to distinguish this tuple from the corresponding point $[\underline{I}(Y)]$ in $\mathbb{P}^3_{1,1,1,2}(K)$. 
This motivates the following definition.

\begin{definition}\label{def:normalise} 
Let $K$ be  a local field with valuation $\nu$, as in Section \ref{sec:intro}. A point $\underline{I}=(I_3,I'_3,I''_3,I_6)\in K^4$ 
is said to be \emph{normalised at $\nu$} if all coordinates have positive valuation and at least one of them has valuation equal to zero.
 A  \emph{normalisation}  $\underline{I}_\nu\in (K')^4$ of $\underline{I}$ is a normalised vector defined over an extension field $K'/K$ that defines the same point in the weighted projective space $\mathbb{P}^3_{1,1,1,2}(K').$ 
\end{definition}

\begin{remark} 
Let $\underline{I}=(I_3, I_3', I_3'', I_6)\in K^4$ be a set of Ciani invariants. It can be normalised over a quadratic extension $K'/K$. Namely, define $m:=\min(\nu(I_3), \nu(I'_3), \nu(I''_3),\nu(I_6)/2)$ and let $u\in K'$ be an element with $\nu(u)=m.$ Then   $\underline{I}_\nu=(I_3/u, I_3'/u, I_3''/u, I_6/u^2)\in {K'}^4$ is normalised. 
Notice that $[\underline{I}]=[\underline{I}_\nu]\in\mathbb{P}^3_{1,1,1,2}(K)$, as required in Definition \ref{def:normalise}. 
\end{remark}

\begin{remark}\label{normI6}
If $\nu(I_6)$ is even or $m\neq \nu(I_6)/2$, then we can normalise $\underline{I}$ at $\nu$ over $K$, so $\underline{I}_\nu\in K^4$. In the case that we cannot normalise at $\nu$ over $K$, we have that after normalisation $\nu(I_3),\nu(I'_3),\nu(I''_3)>0$ and $\nu(I_6)=0$. As discussed later, see Lemmas \ref{lem:badprimes} and \ref{lem:potgoodhyper}, it follows that in this situation, the curve $Y/\overline{K}$ with invariants $\underline{I}$ has bad reduction at $\nu$.

\end{remark}

\begin{remark} Given a point $[\underline{ I}]\in\mathbb{P}_{1,1,1,2}^3(K)$ there exists a (not necessarily normalised) point $\underline{I}'\in K^4$ such that $[\underline{I}]=[\underline{I}']$. For example, consider $[(\sqrt{\pi},\sqrt{\pi},\sqrt{\pi},1)]=[(\pi,\pi,\pi,\pi)]$ for an element $\pi\in K$ with $\nu(\pi)=1$. 
\end{remark}

\section{Models of Ciani quartics}\label{sec:reconst}
In this section, given a set of Ciani invariants $\underline{I}\in K^4$, we construct a standard model $Y_1(\underline{I})$ over an explicit finite extension $L'/K$ such that $[\underline{I}]=[\underline{I}(Y_1(\underline{I}))]$. We also construct a $K$-model of this curve that we call $Y_0(\underline{I})$. We finally characterise when $Y(\underline{I})$ is special in terms of the invariants and we study the twists of $Y_0(\underline{I})/K$.

\subsection{Reconstruction of a Ciani quartic with given invariants}\label{sec:recons}

 In this section and in the following, we exceptionally assume that $K$ is an arbitrary (even global) field of characteristic different from $2$. We can change the assumptions on $K$ as the result works in this generality. Let $\underline{I} = (I_3 , I_3', I_3'', I_6)\in K^4$ be a fixed set of invariants with $\Delta(\underline{I})=2^{20}I_3I_3''^4I_6^2 \neq 0$.

 The goal of this subsection is to construct a Ciani quartic $Y_1(\underline{I})$ given by a standard model over an explicit field extension $L'/K$ such that $[\underline{I}]=[\underline{I}(Y_1(\underline{I}))]$. We also provide a $K$-model $Y_0(\underline{I})$ of this quartic.

We start by introducing some notation, which is used in the next proposition.
Set \(P = 8I_3 + I'_3 - I''_3\), consider the polynomial $\mathcal{P}(T) = T^3 - S_1T^2 + S_2T - S_3$ with
		\begin{equation}\label{eq:coeffP}
  \begin{split}
		S_1= &\,I'_3+12I_3,\\
		S_2= &\,\frac{1}{4}(P^2 + 16I_3(P + I_3'') - I_6), \\
		S_3= &\,I_3P^2,
  \end{split}
		\end{equation}
	and denote the roots of \(\mathcal{P}\) by \(\mathcal{A}, \mathcal{B},\mathcal{C}\).

	\begin{proposition}\label{reconsquartic} Let \(\underline{I} = (I_3, I_3', I_3'', I_6)\in K^4\) be a set of invariants. Assume \(\Delta(\underline{I}) = 2^{20}I_3I_3''^4I_6^2 \neq 0\). 
	Let $L/K$ be the splitting field of the polynomial  $\mathcal{P}$.
	\begin{enumerate}[(a)]
	\item Assume  \(P \neq 0\). The Ciani quartic defined by the standard model
		\begin{equation}\label{eq:modelPnot0} Y_1(\underline{I}):\,\mathcal{A} x^4 + \mathcal{B} y^4 + \mathcal{C} z^4 + P(x^2y^2 + y^2z^2 + z^2x^2) = 0\end{equation}
		has invariants $\underline{I}(Y_1(\underline{I}))=( P^2I_3, P^2I'_3,P^2I''_3, P^4I_6)$. The discriminant is $\Delta(Y_1(\underline{I}))=\Delta(\underline{I}) P^{18}$. The quartic $Y_1(\underline{I})$ is defined over $L'=L=K(\cA,\cB,\cC)$.
	\item Assume that $P=0$ and that $0$ is a simple root of $\mathcal{P}$, hence $S_2\neq 0$.  It is no restriction to  assume that $\cA=0$ and $\cB, \cC\neq 0$.
  The Ciani quartic defined by the standard model
	\begin{equation}\label{eq:modelcaseb}    
     Y_1(\underline{I}):\,I_3 S_2 x^4 + \mathcal{B} y^4 + \mathcal{C} z^4 + S_2 x^2(y^2 + z^2) = 0
     \end{equation}
    has invariants $\underline{I}(Y_1(\underline{I}))=( S_2^2I_3, S_2^2I'_3,S_2^2I''_3, S_2^4I_6)$. The discriminant of this model is $\Delta(Y_1(\underline{I}))=\Delta(\underline{I}) S_2^{18}$ and it is defined over $L'=L=K(\cB,\cC)$.
	\item Assume that $P=0$ and  $0$ is at least a double root of $\mathcal{P}$. It is no restriction  to assume that $\cA=\cB=0$. Choose $r^2=I_3S_1$. The Ciani quartic defined by the standard model
\begin{equation}\label{eq:modelcasec}	
 Y_1(\underline{I}):\,I_3(x^4+y^4+z^4)+rx^2y^2=0
 \end{equation}
	 has invariants $\underline{I}(Y_1(\underline{I}))=( I_3^3, I_3^2I'_3,I_3^2I''_3, I_3^4I_6)$, the discriminant is $\Delta(Y_1(\underline{I}))=\Delta(\underline{I}) I_3^{18}$, and it is defined over the extension $L'=L(r)/K$. 
		\end{enumerate}
	\end{proposition}
	
	\begin{proof} It is straightforward to check that the given models have the given invariants and that they are defined over the given fields.
		\end{proof}

\begin{notation}\label{not:LI}
    Let $\underline{I}\in K^4$ be a set of invariants. We write $L(\underline{I})$ for the field $L$ from Proposition \ref{reconsquartic}, i.e. for the splitting field of the polynomial $\mathcal{P}$. It only depends on $[\underline{I}]\in\mathbb{P}^3_{1,1,1,2}(K)$. In addition, in cases (a) and (b) we have $L'=L$. 
\end{notation}
		
	\begin{remark}\label{rmk:roots}
 Assume that $Y$ is a Ciani quartic given by an arbitrary standard model \eqref{eq:standard_model}. Take $\underline{I}=\underline{I}(Y)$.
 \begin{enumerate}[(a)]
\item 
 If $P=8I_3+I'_3-I''_3=abc\neq 0$, then $Y_1(\underline{I})$ is given by \eqref{eq:modelPnot0}, and we may take
 \[
\cA=Aa^2, \quad \cB=Bb^2, \quad \cC=Cc^2.
  \]
  \item If $P=0 $, the polynomial $\mathcal{P}$ has $0$ as a root.  Assume that  $0$  is a simple root of  $\mathcal{P}$. In this case $Y_1(\underline{I})$ is given by \eqref{eq:modelcaseb}, where we  take
  \[
  \cB=Bb^2, \qquad \cC=Cc^2.
  \]
     \item Assume that $P=0$ and that $0$ is a root of $\mathcal{P}$ of multiplicity at least $2$. In this case $Y_1(\underline{I})$ is  given by \eqref{eq:modelcasec} with 
  \[
  r^2=ABC^2c^2.
  \]
  Note that in this case  $Y_1(\underline{I})$ is defined over the extension $L(r)/L$ of degree at most $2$.
  \end{enumerate}
	\end{remark}

The Ciani quartics $Y_1(\underline{I})$ can be descended to $K$. The next result provides a particular $K$-model that we call $Y_0(\underline{I})$. Later on, in Lemma \ref{lem:twists}, we will give all the $K$-models of $Y_1(\underline{I})$. In order to do that we will compute the twists of $Y_0(\underline{I})/K$.

\begin{proposition}\label{prop:reconsK}
With the notation in Proposition \ref{reconsquartic} and the data in Table \ref{tab:reconsK}, the morphism $\phi:\, Y_0(\underline{I})\rightarrow Y_1(\underline{I})$ defines a $K$-model $Y_0(\underline{I})$ with $\underline{I}(Y_0(\underline{I}))=(\lambda I_3(Y_1(\underline{I})),\lambda I'_3(Y_1(\underline{I})),\lambda I''_3(Y_1(\underline{I})),\lambda^2I_6(Y_1(\underline{I})))$.
\begin{table}[] \centering
\begin{tabular}{|c|ccc|}
\hline
          & \multicolumn{3}{c|}{$(a)$}                            \\ \hline
{$[L:K]$} & \multicolumn{1}{c|}{1} & \multicolumn{1}{c|}{2} & 3    \\ \hline
$\phi$    & \multicolumn{1}{c|}{$\begin{pmatrix}1 & 0 & 0 \\ 0 & 1 & 0 \\ 0 & 0 & 1\end{pmatrix}$}  & \multicolumn{1}{c|}{$\begin{pmatrix}1 & 0 & 0 \\ 0 & 1 & \mathcal{B}  \\ 0 & 1 & \mathcal{C} \end{pmatrix}$}  &  $\begin{pmatrix}1 & \mathcal{A} & \mathcal{A}^2 \\ 1 & \mathcal{B} & \mathcal{B}^2 \\ 1 & \mathcal{C} & \mathcal{C}^2\end{pmatrix}$  \\ \hline
$\lambda$ & \multicolumn{1}{c|}{1}  & \multicolumn{1}{c|}{$(\mathcal{B}-\mathcal{C})^4$}  & $Q^4$  \\ \hline
\end{tabular}
\end{table}

\begin{table}[] \centering
\begin{tabular}{|c|cc|c|}
\hline
      &    \multicolumn{2}{c|}{$(b)$}   & $(c)$ \\ \hline
{$[L:K]$} & \multicolumn{1}{c|}{1} & 2 & 1   \\ \hline
$\phi$    & \multicolumn{1}{c|}{$\begin{pmatrix}1 & 0 & 0 \\ 0 & 1 & 0 \\ 0 & 0 & 1\end{pmatrix}$}  & $\begin{pmatrix}1 & 0 & 0 \\ 0 & 1 & \mathcal{B}  \\ 0 & 1 & \mathcal{C} \end{pmatrix}$ &   $\begin{pmatrix}\sqrt{r} & 0 & 0 \\ 0 & 1 & 0  \\ 0 & 0 & 1 \end{pmatrix}$  \\ \hline
$\lambda$ &\multicolumn{1}{c|}{1}  & $(\mathcal{B}-\mathcal{C})^4$  &   ${I_3S_1}$  \\ \hline
\end{tabular}
\caption{Explicit isomorphisms for $K$-models of Ciani quartics.}
\label{tab:reconsK}
\end{table}
\end{proposition}

\begin{proof} It is straightforward to check that the curves $Y_0(\underline{I})$ defined by the isomorphisms $\phi$ are defined over $K$. The equality relating the invariants of both models is a consequence of Lemma \ref{trans}.
\end{proof}

\subsection{Characterisation of special Ciani quartics}\label{sec:special}

As in the previous section, we again assume that $K$ is any (possibly global) field of characteristic $\neq 2$. Fix a set of Ciani invariants $\underline{I}\in K^4$. In this section, we characterise special Ciani quartics in terms of their Ciani invariants. We refer to \cite{LRRS} for a similar result characterising the automorphism group of an arbitrary smooth quartic in terms of its Dixmier--Ohno invariants.

We write $Q$ for the discriminant of the polynomial $\mathcal{P}$ introduced in Section \ref{sec:recons}. A straightforward calculation shows that 
\begin{equation}\label{eq:Qdef}
Q=-4I_3I_3'^3I_6 - 27I_3^2I_6^2 + 18I_3I_3'I_6I + I_3'^2I^2 - 4I^3.
\end{equation}

\begin{lemma}\label{lem:Q}  The Ciani quartic $(Y(\underline{I}), V)$ with invariants $\underline{I}$ is special if and only if $Q=0$. 
\end{lemma}

\begin{proof} Being special is preserved under $\bar{K}$-isomorphism. We can therefore check the condition for a standard model $Y_1(\underline{I})$ as in Proposition \ref{reconsquartic}.

In case (c) of Proposition \ref{reconsquartic} the curve $Y_1(\underline{I})$ is special and $Q=0$. Therefore it remains to consider the cases (a) and (b) of Proposition \ref{reconsquartic}, and the statement is automatically satisfied. 

The condition $Q=0$ implies that the polynomial $\mathcal{P}$ has a double root. In  case (b) of Proposition \ref{reconsquartic} we have $\cB=\cC$. This implies that $\tau(x:y:z)\mapsto (x:z:y)$ is an automorphism of $Y_1(\underline{I})$ that is not contained in the fixed Ciani subgroup $V$.  In case (a) the curve $Y(\underline{I})$ admits a similar automorphism permuting two of the variables, depending on which two roots are equal.

 We prove the reverse implication.  Assume that $Y(\underline{I})$ is special. From the classification of automorphism group of plane quartics, it follows that $\Aut_{\overline{L}}(Y(\underline{I}))$ contains the dihedral group $D_4$ of order $8$ as subgroup, see \cite{LRRS} or Figure \ref{fig:stratification}. Using the explicit description of these curves in  Examples \ref{exa:fieldD4} and \ref{exa:fieldD4C} it is easy to check that for any quartic in this family we have $Q=0$. 
\end{proof}

\subsection{Twists of non-special Ciani quartics}\label{sec:twist}
From this section on, we assume that $p>3$ is a prime and that $K$ is a finite extension of $\Q_p^\nr$. Write $\nu$ for the valuation on $K$. When working over an extension field $K'/K$, we always extend $\nu$ to $K'$. 
Let $\underline{I}\in K^4$ be a set of invariants such that $Y(\underline{I})$ is non-special. Let $L=L(\underline{I})/K$ be the minimal finite extension over which $Y_1(\underline{I})$ is defined, as in Proposition \ref{reconsquartic}. Since $K$ is a local field, we have that $L/K$ is Galois and $[L:K]\leq 3$. We write $Y_0(\underline{I})$ for the $K$-model of $Y(\underline{I})$ we constructed in Proposition \ref{prop:reconsK}.

In this section, we describe the $K$-twists of $Y_0(\underline{I})$. 
The results are a special case of the result of \cite{elg-twists}, presented in a way we will use afterwards.

The group $\Gamma_K:=\Gal(\overline{K}/K)$ acts on $\Aut(Y_0(\underline{I}))=V\simeq C_2^2$, with action induced by that of $\Gal(L/K)$ on the roots of $\mathcal{P}$. If $[L:K]=1$, the action is trivial. In the case  that $[L:K]=2$, exactly two of the roots of $\mathcal{P}$ are congruent modulo $\nu
$, i.e.
$\cB\equiv \cC\pmod{\nu}$. Then the generator $\tau\in\Gal(L/K)$ acts on $V$ by ${\,}^\tau\sigma_b=\sigma_c,$ where we use the notation from Notation \ref{not:sigma}. Similarly, if $[L:K]=3,$ all three roots of $\mathcal{P}$ are congruent modulo $\nu$, and  the generator of $\Gal(L/K)$ acts by cyclically permuting the nontrivial elements of order $2$ in $V$. 

The following lemma determines the $K$-models of $Y(\underline{I})$ up to $K$-isomorphism, that is the set of twists $\operatorname{Twist}(Y_0(\underline{I})/K)$ of $Y_0(\underline{I})$.

\begin{lemma}\label{lem:twists}
    Let $\underline{I}\in K^4$ and $Y_0(\underline{I})$, defined as in Proposition \ref{prop:reconsK}, be  non-special. Let $L$ be the splitting field of the polynomial $\mathcal{P}$ introduced in Section \ref{sec:recons}. Then
    $$\#\operatorname{Twist}(Y_0(\underline{I})/K)=\begin{cases}4 \text{ if }[L:K]=1,\\
    2 \text{ if }[L:K]=2,\\
    1 \text{ if }[L:K]=3.
    \end{cases}$$

   In addition:
          \begin{enumerate}[(a)]
    \item If $[L:K]=1$, then the Ciani quartics $Y'_0(\underline{I})/K$  given by the isomorphisms $$
\begin{pmatrix}1 & 0 & 0 \\ 0 & 1 & 0 \\ 0 & 0 & 1\end{pmatrix},\begin{pmatrix}\pi & 0 & 0 \\ 0 & 1 & 0 \\ 0 & 0 & 1\end{pmatrix},\begin{pmatrix}1 & 0 & 0 \\ 0 & \pi & 0 \\ 0 & 0 & 1\end{pmatrix},\begin{pmatrix}1 & 0 & 0 \\ 0 & 1 & 0 \\ 0 & 0 & \pi\end{pmatrix}:\, Y'_0(\underline{I})\rightarrow Y_0(\underline{I}),
$$
correspond to different elements in $\operatorname{Twist}(Y_0(\underline{I})/K)$. Here $\pi\in K'$ is an element with $\nu(\pi)=1/2$ in a quadratic extension $K'/K$.

\item If $[L:K]=2$, then the Ciani quartics $Y'_0(\underline{I})/K$ given by the isomorphisms
$$
\begin{pmatrix}1 & 0 & 0 \\ 0 & 1 & \mathcal{B}  \\ 0 & 1 & \mathcal{C} \end{pmatrix},\begin{pmatrix}\pi_1 & 0 & 0 \\ 0 & 1 & \mathcal{B}  \\ 0 & \zeta_4 & \zeta_4\mathcal{C} \end{pmatrix}:\, Y'_0(\underline{I})\rightarrow Y_0(\underline{I}),
$$
correspond to different elements in $\operatorname{Twist}(Y_0(\underline{I})/K)$. Here $\pi_1\in K''$ is an element with $\nu(\pi_1)=1/4$ in an extension $K''/K$ and $\zeta_4$ is a primitive $4$-th root of unity. 
    \end{enumerate}
        \end{lemma}

\begin{proof} This is a consequence of \cite[Prop.~5.2]{Twistsg3}, but we include the proof for the sake of completeness. In order to compute $\operatorname{Twist}(Y_0(\underline{I})/K)=H^1(\Gal(\bar{K}/K), V)$ we just notice that every cocycle splits (in the sense of \cite[Def. 2.1]{elg-twists}) over a Galois extension  with Galois group a subgroup of $V\rtimes\operatorname{Gal}(L/K)$, see \cite[Sec.~4, Step 2]{elg-twists}. In our case, this gives a finite number of cyclic extensions. A cocycle is then determined by the image of a generator  of the Galois group in $V$. Two cocycles are equivalent if they are conjugated as in Equation $(2)$ in \cite{elg-twists}. This is again a finite number of checks. We compute the numbers given in the statement of the lemma. 
It is straightforward to check that the isomorphisms $\phi$ that we give define non-equivalent cocycles in $H^1(\Gal(\bar{K}/K), V)$ by the rule:  $\xi(\sigma)=\phi\circ^\sigma\phi^{-1}$ for each $\sigma\in\Gal(\bar{K}/K)$. 
\end{proof}

\section{Stable reduction and conductor exponents}\label{sec:stablered}

In this section, we collect some results on the stable reduction of Ciani quartics. Let $K$ be as in the previous section. We write $k$ for the residue field of $K$. In the rest of the paper, we will mostly assume that $K=\Q_p^{\nr}$ as this is the case that is most interesting for our purposes.

\begin{definition}\label{def:goodred}
Let $Y/K$ be a smooth projective and absolutely irreducible plane quartic.  
\begin{itemize}
    \item[(a)] The curve $Y$ has \emph{good reduction} if there exists a flat and proper $\mathcal{O}_K$-scheme $\mathcal{Y}$ with generic fiber $Y$ such that the special fiber $\overline{Y}:=Y\otimes_{\mathcal{O}_K}k$, which we call the \emph{reduction} of $Y$, is smooth.  
    \item[(b)] The curve $Y$ has \emph{good quartic reduction} if $Y$ has good reduction and its reduction $\overline{Y}$ is also a plane quartic. If $Y$ has good but not good quartic reduction, we say that it has \emph{good hyperelliptic reduction}.
    \item[(c)]  We say that $Y$ has \emph{potentially} good quartic (resp. hyperelliptic) reduction if $Y$ has good quartic (resp. hyperelliptic) reduction after replacing $K$ by a finite extension. 
    \item[(d)] Otherwise, we say that $Y$ has \emph{geometrically bad reduction}. 
\end{itemize}
\end{definition}

We note that, if $Y$ has good quartic reduction over $K$, then the $\mathcal{O}_K$-scheme from Definition \ref{def:goodred} may be defined by a homogeneous polynomial $F\in \mathcal{O}_K[x,y,z]$ of degree $4$ whose reduction $\overline{F}$ modulo the uniformising element $\pi$ of $K$ is a quartic with $\Delta(\bar{F})\neq 0$.

\begin{lemma}\label{lem:badprimes}
    \label{lem:hypo} Let $\underline{I}\in K^4$. A Ciani quartic  $Y$ with $[\underline{I}(Y)]=[\underline{I}]$ has potentially good quartic reduction at $\nu$ if and only if $\nu(\Delta(\underline{I}_\nu))=0$. 
\end{lemma}

\begin{proof}
The statement follows from \cite[Prop.~6]{BCKKLS}. 
\end{proof}

In Theorem 2 of \cite{BCKKLS} one finds a characterisation of the type of stable reduction of $Y_{\underline{I}}$ at $\nu$ in terms of the invariants in the case that $Y$ has geometrically bad reduction at $\nu$. A Ciani quartic with $\nu(\Delta(\underline{I}_\nu))>0$ may still have potentially good reduction at $\nu$. However, in this case, the reduction of $Y$ is hyperelliptic. 

\begin{lemma}\label{lem:potgoodhyper}
Let $Y$ be a Ciani plane quartic with invariants $\underline{I}$, which we assume to be normalised, i.e. 
 $\underline{I}_{\nu}=\underline{I}$. Then $Y$ has potentially good hyperelliptic reduction at $\nu$ if and only if 
\[
0<3\nu(I''_3)=2\nu(I_6)\leq 6\nu(I'_3).
\]
\end{lemma}

\begin{proof} This is a special case of \cite[Theorem 3]{BCKKLS}.
\end{proof}

Let $Y/K$ be a curve with potentially good reduction. It follows from the Stable Reduction Theorem, proved by Deligne--Mumford (\cite{DM}), and the assumption that the residue field $k$ of $K$ is algebraically closed, that there is a unique minimal extension $M/K$ over which $Y_M:=Y\otimes_K M$ has good reduction.  The uniqueness of $M$ implies that the extension $M/K$ is Galois. We write $G:=\Gal(M/K)$. 
Since $g(Y)=3\geq 2$, there exists a unique smooth model $\mathcal{Y}$ of $Y_M$ over $\mathcal{O}_M$. 

The Galois group $G$ acts faithfully and $k$-linearly on the special fiber $\overline{Y}$ of $\mathcal{Y}$. We obtain an embedding
\[
G\hookrightarrow \Aut_k(\overline{Y}).
\]

We define the \emph{inertial reduction of $Y/K$} as the curve $\overline{Z}:=\overline{Y}/G.$

An important arithmetic invariant associated with a curve $Y$ over a local field is its \emph{conductor exponent} $f_p=f_p(Y)$. It is a number that gives information on the reduction type of the Jacobian of $Y$, and on the Galois representation associated to it. We refrain from giving a precise  definition, for which we refer to \cite{BW}, but instead we give a formula for computing it in our situation.
Since we assume that $p>3$, it follows from Corollary \ref{cor:boundLK} that the conditions of the next proposition are satisfied in our situation.

\begin{proposition}\label{prop:fp}
 Let $p>2$ be a prime and $K$ be a finite extension of $\mathbb{Q}_p^{\nr}.$ Let $Y/K$ be a smooth projective curve with good reduction over a tamely ramified extension $M/K$.   Then the conductor exponent of $Y$ satisfies
 \begin{equation}\label{eq:fp}
 f_p(Y)=2g(Y)-2g(\overline{Z}).
 \end{equation}
 \end{proposition}

\begin{proof}
 This is a very special case of    \cite[Theorem 1.1]{BW}.
\end{proof}

The following general result follows  from Proposition \ref{prop:fp}.

\begin{corollary}\label{cor:conductor_good}
Let  $K$ be a field containing $\mathbb{Q}_p^{\nr}$ and let $Y/K$ be a Ciani quartic. 
The following are equivalent:
\begin{itemize}
    \item[(a)]  $Y$ has good reduction,
    \item[(b)]  $f_p=0$.
\end{itemize}

\end{corollary}

\begin{proof} 

The forward implication follows from Proposition \ref{prop:fp}. For the converse, it follows from the N\'eron--Ogg--Shaferevich criterion (\cite[Theorems 1,2]{ST68}) that the Jacobian of $Y$ has good reduction over $K$. 
In \cite{BCKKLS} we proved that $Y$ has good reduction, as well.  Namely, in the list of cases in \cite[Theorems 2 and 3]{BCKKLS}, all possibilities for the stable reduction of a Ciani quartic in the case of geometrically bad reduction have at least one loop, i.e.~are not of compact type, which implies that the corresponding Jacobian has bad reduction. 
\end{proof}

Let $(Y, V)/K$ be a Ciani quartic.
Let $L/K$ be a minimal Galois extension such that $Y$ admits a standard model $Y_1$ over $L$. Recall from Lemma \ref{lem:fod_standard} that $\Gal(L/K)\subset S_3$.

\begin{proposition}\label{lem:stable_model}
The curve $Y$ admits stable reduction over an extension $M/L$ with $[M:L] \mid 4$.
\end{proposition}

\begin{proof}
Let $Y_1$ be a standard model of $Y$ over $L$. We use the notation from \eqref{eq:standard_model}.  Denote by $J$  the Jacobian of $Y_1$. By Raynaud's Criterion \cite[Proposition 4.7]{SGA7}, $Y_1$ and $J$ acquire stable reduction over the extension $L(J[m])$, where $m$ is any integer with $m\geq 3$ and $p \nmid m$. We choose $m=4$, and set $M=L(J[4])$. 

\textbf{Claim}: The degree of  $L(J[2])/L$ is at most $2$. 

As in Notation \ref{not:sigma}, we write $E_i=Y_{\overline{K}}/\langle \sigma_i\rangle$ for $i\in 
\{a,b,c\}.$ Then $g(E_i)=1$. Set $\operatorname{A}=E_a\times E_b\times E_c$. We obtain an isogeny $\iota:\operatorname{A}\to J$ over $\bar{K}.$  It follows from \cite[Section 4.1]{HLP} that the kernel of $\iota$ is a subgroup of $\operatorname{A}[2]$.   We show that $[L(\operatorname{A}[2]):L]\leq 2.$ The claim follows, as $\iota$ may be defined over $L(\operatorname{A}[2]).$ Moreover, $J[2]$ is the image of $\operatorname{A}[2]$ under $\iota$. Hence $L(J[2])\subset L(\operatorname{A}[2]).$   

The $2$-torsion points of the $E_i$, and hence of $\operatorname{A}$, are easily described explicitly. Namely, they are the points on $E_i$  above $D\setminus D_i$, where we use the notation from the proof of Lemma \ref{lem:fod_standard}. In other words, $L(\operatorname{A}[2])$ is the field over which all $6$ branch points of the cover $\varphi:Y_1\to Y_1/V=:X_1$ are rational. 
In terms of the coefficients of the standard model \eqref{eq:standard_model}, $L(\operatorname{A}[2])$ is the field obtained by adjoining the roots of the polynomials
\[
p_a(T) = T^2-2aT+4BC, \quad p_b(T) = T^2 -2bt +4AC, \quad p_c(T)=T^2-2cT+4AB.
\]
see for example \cite[Section 2.1]{BCKKLS}.
The claim follows, since $L(\operatorname{A}[2])/L$ is a totally ramified extension of  local fields of characteristic different from~$2$.

It is well known that $L(J[4])/L(J[2])$ is an abelian extension of exponent~$2$. Therefore, the degree of this extension is at most $2$, as well, and $[L(J[4]):L]$ divides $4$. The statement of the lemma follows.
\end{proof}

\begin{corollary}\label{cor:boundLK}
Every Ciani quartic $Y/K$ has good reduction over a cyclic extension $M/K$ of   degree  dividing either $8$ or $12$.
\end{corollary}

\begin{proof}
Let $M/K$ be as in the proof of Proposition \ref{lem:stable_model}.
The statement on the degree of $M/K$ follows immediately from Lemma \ref{lem:fod_standard} and Proposition \ref{lem:stable_model},  since the residue field $k$ of $K$ is algebraically closed and of characteristic $p\geq 5$.   Since $M/K$ is totally ramified and its degree  $p$ does not divide $[M:K]$ it follows that $M/K$ is Galois, and that its Galois group is cyclic.
\end{proof}

\section{Reduction of Ciani quartics}\label{sec:badred}

In this section, we study Ciani quartics with potentially good reduction to characteristic $p> 3$. We only treat the case where $Y(\underline{I})$ is non-special (Definition \ref{def:special}). It is possible to treat the special case analogously, but for the sake of brevity we decided not to include it. We treat the two cases of potentially good quartic and hyperelliptic reduction separately.

Fix a prime $p>3$ and let $K=\Q_p^{\nr}$.  We write $\nu$ for the valuation on $K$ with $\nu(p)=1$. The arguments also work if $K$ is a finite extension of $\Q_p^{\nr},$ though one has to adapt the statements a bit.

Let $\underline{I}=(I_3,I'_3,I''_3,I_6)\in  K^4$ be a set of invariants. We assume that the curve $Y(\underline{I})/\overline{K}$ is smooth and non-special. Moreover, we assume that the invariants $\underline{I}$ are normalised. As we are only interested in the case of potentially good reduction, this is no restriction by Remark \ref{normI6}.

\subsection{The case of potentially good quartic reduction}\label{sec:smoothmodelquartic}
In this section, we assume that $p>3$ is a prime at which $Y(\underline{I})$ has potentially good quartic reduction.  Recall from Lemma \ref{lem:badprimes} that this happens if and only if $p\nmid \Delta(\underline{I}_\nu)$. Since $\underline{I}$ is normalised, it follows that $\nu(I_3)=\nu(I''_3)=\nu(I_6)=0.$ We set $v:=\nu(I'_3).$ This is the only coefficient of $\underline{I}$ whose valuation may be positive in our situation. 
 
In Proposition \ref{reconsquartic} we constructed an explicit field extension $L=L(\underline{I})/K$ together with a standard model $Y_1(\underline{I})$ over $L$.
The discriminant $\Delta(Y_1(\underline{I}))$ of the model $Y_1(\underline{I})$ differs from $\Delta(\underline{I})$ by a factor. In Proposition \ref{prop:reconsK} we constructed an explicit model $Y_0(\underline{I})$ over $K$ of $Y_1(\underline{I})$. The discriminant $\Delta(Y_0(\underline{I}))$ differs from $\Delta(Y_1(\underline{I}))$ by another factor. Therefore, $Y_0(\underline{I})$ does not need to have good reduction over $K$.  In the following propositions, we analyse whether $Y_0(\underline{I})$ or any of its twists have good reduction over $K$.

\begin{proposition} \label{prop:smoothmodel_new} Let $\underline{I}=(I_3,I'_3,I''_3,I_6)\in  K^4$ be a set of normalised invariants. 
Assume that $\nu(\Delta(\underline{I}))=\nu(Q)=0$ and that $Y(\underline{I})$ is non-special.  There exists a standard Ciani model $Y_2/K$ with good reduction at $p$. This model satisfies $f_p(Y_2)=0.$ For the other twists $Y'_2$ of $Y_2$, one has $f_p(Y'_2)=4$.
\end{proposition}

\begin{proof}
The assumption that $\nu(Q)=0$ implies that the  polynomial $\mathcal{P}$ from Proposition \ref{reconsquartic} splits in $K$. In particular, the standard model $Y_1(\underline{I})$ of $Y(\underline{I})$ from that proposition may be defined over $L=K$. 

If $P\neq0$, we are in case (a) of Proposition \ref{reconsquartic}, and the discriminant of this standard model is  $\Delta(Y_1(\underline{I}))=\Delta(\underline{I})P^{18}$. The assumption that $Y_1(\underline{I})$ has potentially good reduction implies that $p\nmid \Delta(\underline{I})$. If $\nu(P)=0$ it follows that $Y_1(\underline{I})$ already has good reduction. We take $Y_2(\underline{I})=Y_1(\underline{I})$ in this case. 
If $\nu(P)>0$, the assumption that $\nu(Q)=0$ implies that only one of the roots $\mathcal{A}, \mathcal{B}, \mathcal{C}$ of $\mathcal{P}$ defined in Equation \ref{eq:coeffP} has positive valuation. Let us say that it is $\mathcal{A}$. Moreover, we have from that equation that
\begin{equation}\label{eq:valP}
2\nu(P)=\nu(\mathcal{A})+\nu(\mathcal{B})+\nu(\mathcal{C}).
\end{equation}

So $\nu(\mathcal{A})$ is even, and dividing $x$ by $p^{\nu(\mathcal{A})/4}$ we obtain another standard model $Y_2(\underline{I})$, still defined over $K$, whose discriminant satisfies
$$
    \nu(\Delta(Y_2))=\nu(\Delta(Y_1(\underline{I}))) + 36(\nu(\cA))/4
    =
\nu(\Delta(Y_1(\underline{I})))-18\nu(P)=\nu(\Delta(\underline{I})).
$$
Hence, the model $Y_2$ has good reduction over $K$.

If $P=0$, i.e.~we are in case (b) of Proposition \ref{reconsquartic}, we argue similarly: we can make $\mathcal{A}=0$ and then $\Delta(Y_1(\underline{I}))=\Delta(\underline{I})S_2^{18}$ with $S_2=\mathcal{B}\mathcal{C}$. Moreover, $Q=\mathcal{B}^2\mathcal{C}^2(\mathcal{B}-\mathcal{C})^2$. The assumption $\nu(Q)=0$ implies then that $Y_1(\underline{I})$ already has good reduction. In this case, we may therefore choose $Y_2(\underline{I})=Y_1(\underline{I})$.

We do not need to consider case (c) of Proposition \ref{reconsquartic}, since $Y(\underline{I})$ is special in this case.

The claim for the twists follows from Lemma \ref{lem:twists} (if $[L:K]=1$) and Proposition \ref{prop:fp}; see the proof of the Proposition \ref{prop:val_equal} for more details on the computation of the conductor in a more complicated case.
\end{proof}

\begin{remark} Another way of proving that the other twists have positive conductor, i.e. that they are not $K$-isomorphic to a model with good reduction, is via the Elsenhans--Stoll minimal model reduction algorithm in \cite{Elsenhans}. This approach is computationally more expensive and does not provide the value of the conductor of the twists that do not have good reduction.
\end{remark}

We now study the case in which $\nu(Q)>0$.

\begin{proposition}\label{prop:val_equal} We use the notation from Proposition \ref{reconsquartic}.
Let $p\nmid \Delta(\underline{I})$ be a prime of potentially good quartic reduction.
Assume that $\nu(Q)>0$. 
\begin{enumerate}[(a)]
\item  The $K$-model $Y_0(\underline{I})$ of $Y(\underline{I})$ given in Proposition \ref{prop:reconsK} has good reduction over an extension of $L(\underline{I})$ degree at most $2$.  The conductor is $f_p(Y_0(\underline{I}))=~4$.
\item  If $Y(\underline{I})$ is non-special, all $K$-models $Y'_0$ of $Y(\underline{I})$ have $f_p(Y'_0)= 4$.
\end{enumerate}
\end{proposition}

\begin{proof} The proof proceeds in the following steps. In step 1 we define a standard model $Y_2(\underline{I})$ over an extension $L_8$ of $L$ with good reduction. 
This model is defined by giving an explicit isomorphism $\psi:\,Y_1(\underline{I})\rightarrow Y_2(\underline{I})$.  Composing $\psi$ with the isomorphism $\phi$ from Proposition \ref{prop:reconsK}, we get an isomorphism from $Y_0(\underline{I})$ to $Y_2(\underline{I})$ defined over $L_2$. In step 2 we compute the action of $\Gal(L_8/K)$ on the special fiber $\overline{Y}$ of $Y_2(\underline{I})$ and compute the conductor exponent $f_p(Y_0(\underline{I}))$ via Proposition \ref{prop:fp}. In step 3 we proceed in the same way with the other $K$-twists of $Y_0(\underline{I})$ described in Lemma \ref{lem:twists}.

\bigskip
 We assume we are in case (a) of Proposition \ref{reconsquartic} and that $[L:K]=2$.  The other cases are easier, so we discuss these at the end of the proof.  By assumption, $P\neq 0$. As in the proof of Proposition \ref{prop:smoothmodel_new}, the roots $\mathcal{A}, \mathcal{B}, \mathcal{C}$ of the polynomial $\mathcal{P}$ from Proposition \ref{reconsquartic} satisfy \eqref{eq:valP}: 
\begin{equation*}
2\nu(P)=\nu(\mathcal{A})+\nu(\mathcal{B})+\nu(\mathcal{C}).
\end{equation*}

\emph{Step 1}: Without loss of generality, we may assume  $\cA\in K$ and $\cB$, $\cC$ conjugated in $L\setminus K$. In particular, $\nu(\cB)=\nu(\cC)$. Let $L_8:=K(p^{1/8})/K$ be the ramified  extension of degree $8$.
We consider the isomorphism:
\begin{equation}\label{eq:psi}
\psi=\begin{pmatrix}p^{\nu(\cA)/4} & 0 & 0\\ 0 & p^{\nu(\cB)/4} & 0 \\ 0 & 0 & p^{\nu(\cC)/4} \end{pmatrix}:\, Y_1(\underline{I})\rightarrow Y_2(\underline{I}),
\end{equation}
where $Y_2(\underline{I})/L_8$ is a model of $Y_1(\underline{I})$ defined by the equation
\begin{equation}\label{eq:modelY'}
\frac{\mathcal{A}}{p^{\nu(\cA)}}x_2^4+\frac{\mathcal{B}}{p^{\nu(\cB)}}y_2^4+\frac{\mathcal{C}}{p^{\nu(\cC)}}z_2^4+\frac{P}{p^{(\nu(\cA)+\nu(\cB))/2}}\left(x_2^2y_2^2+ x_2^2z_2^2\right)+ \frac{P}{p^{\nu(\cB)}} y_2^2z_2^2=0.\end{equation} 
It follows from \eqref{eq:valP} that the coefficients of this equation are integral. 
A direct calculation using Lemma \ref{trans} and the expression for $\Delta(Y_1(\underline{I}))$ in Proposition \ref{reconsquartic}.(a) yields
\begin{equation}\label{eq:Delta2}\begin{split}
    \nu(\Delta(Y_2(\underline{I})))&=\nu(\Delta(Y_1(\underline{I}))) - 36(\nu(\cA)+\nu(\cB)+\nu(\cC))/4\\
    &=
\nu(\Delta(Y_1(\underline{I})))-18\nu(P)=\nu(\Delta(\underline{I})).
\end{split}
\end{equation}
It follows that the model $Y_2(\underline{I})$  has good reduction over $L_8$.
In fact, the matrix defining $\psi$ in \eqref{eq:psi} should be considered as element of $\PGL_3(L_8)$ and can therefore be divided by 
\[
p^{\min(\nu(\cA)/4, \nu(\cB)/4)}.
\]
Checking the possibilities for the minimum separately,  one sees that $\psi$ is defined over the subextension $L_4$ with $[L_4:K]=4.$  It follows that the model $Y_2(\underline{I})$ has good reduction over $L_4$.

    \bigskip
    \emph{Step 2}: To compute the conductor of $Y_0(\underline{I})$ via Proposition \ref{prop:fp}, we need to determine the action of the Galois group $\Gal(L_8/K)$ on the reduction $\overline{Y}$ of $Y_2(\underline{I})$.  We choose a generator $\sigma$ of $\Gal(L_8/K)$ with $\sigma(p^{1/8})=\zeta_8 p^{1/8}$ for a fixed primitive $8$-th root of unity $\zeta_8\in K$.

Composing $\psi$ with the isomorphism $\phi:Y_0(\underline{I})\to Y_1(\underline{I})$ from Proposition \ref{prop:reconsK}, we obtain an isomorphism $\psi\circ\phi:\,Y_0(\underline{I})
\to Y_2(\underline{I})$ given by
\begin{equation}\label{eq:descent}
\begin{pmatrix} x_2\\y_2\\z_2\end{pmatrix}=\begin{pmatrix}  p^{\nu(\cA)/4}&0&0\\ 0&p^{\nu(\cB)/4}&p^{\nu(\cB)/4}\cB\\0&p^{\nu(\cB)/4}&p^{\nu(\cB)/4}\cC\end{pmatrix}\begin{pmatrix} x_0\\y_0\\z_0\end{pmatrix}.
\end{equation}

Recall that $Y_0(\underline{I})$ is defined over $K$. Using that the functions $x_0, y_0, z_0$ are defined over $K$, we have
\begin{equation}\label{eq:inertia}
\begin{split}
\sigma^\ast(x_2,y_2,z_2)&=
\left( \sigma(p^{\nu(\cA)/4}) x_0,  \sigma(p^{\nu(\cB)/4}) \left(y_0+\sigma(\cB)z_0\right), \sigma(p^{\nu(\cB)/4}) \left(y_0+\sigma(\cC)z_0\right)\right)\\
&= (\zeta_8^{2\nu(\cA)}x_2, \zeta_8^{2\nu(\cB)}z_2, \zeta_8^{2\nu(\cB)} y_2).
\end{split}
\end{equation}
More details on this calculation in a slightly different situation can be found in \cite[Section 5]{BW}. 
Note that $\nu(\cB)=\nu(\cC)\in (1/2)\mathbb{Z}$, so we use the integer $2\nu(\cB)$ in \eqref{eq:inertia}.

It follows that  the automorphism of the stable reduction $\overline{Y}$ induced by $\sigma$ is
\[
\sigma(x_2: y_2:z_2)=(x_2:\alpha z_2:\alpha y_2) \qquad \text{ for some $\alpha\in k$ with $\alpha^4=1$.}
\]
 We compute that $g(\overline{Y}/\langle \sigma\rangle)=1$, independent of whether the order of $\sigma$ is two or four. By Proposition \ref{prop:fp} we find $f_p(Y_0(\underline{I}))=4.$

\emph{Step 3:} We proceed similarly with the quadratic twist $Y'_0$ of $Y_0(\underline{I})$ as in Lemma \ref{lem:twists}, and conclude that $f_p(Y'_0)=4$, as well. This proves the proposition in  case (a) of Proposition \ref{reconsquartic} if $[L:K]=2$.  

\bigskip
We omit the proof for case (b) of Proposition \ref{reconsquartic} and $[L:K]=2$, since it is very similar, except that additionally we have  $\cA=0.$ 

It remains to consider the case $[L:K]=3$. We are automatically in case (a) of Proposition \ref{reconsquartic}. The argument is again similar, but in this case $\sigma$ acts on $\overline{Y}$ by cyclically permuting the variables  $x_2, y_2, z_2$. One computes that $g(\overline{Y}/\langle \sigma\rangle)=1$ in this case, as well. This finishes the proof of the proposition.
\end{proof}

\begin{remark} \label{rem:aut}
\begin{enumerate}[(a)] 
\item In the case that $L/K$ is ramified with ramification index $e_p=2$ in Proposition \ref{prop:val_equal}, we see that $\sigma\in \Gal(L/K)$ acts on the reduction $\overline{Y}$ as an automorphism  of order $2$ or $4$ that is not contained in the fixed Ciani subgroup. This is consistent with the statement of Lemma \ref{lem:Q}, as $\nu(Q)>0$ implies that the reduction $\overline{Y}$ is special, and hence that $\Aut_k(\overline{Y})$ contains at least the dihedral group $D_4$ of order $8$ as subgroup.

In the case that $L/K$ is ramified with ramification index $e_p=3$, the automorphism group $\Aut_{k}(\overline{Y})$ contains an element of order $3$. It follows that $S_4\subset \Aut_{k}(\overline{Y})$ in this case. 
\item In the proof of Proposition \ref{prop:val_equal}, we showed that $Y_1(\underline{I})$ acquires stable reduction over an extension of $K$ of degree  $4$ by looking at the equation of $Y_2(\underline{I}).$ One can also deduce this from the fact that the action of $\Gal(L_8/K)\simeq C_8$ on the reduction $\overline{Y}$ acts via a group of order at most $4$.
\end{enumerate}
\end{remark}

\begin{example}\label{exa:fp_small}
    We finish this section with a concrete example. We choose $I_3=I_3''=I_6=1, I_3'=-6$. Using the notation from Proposition \ref{reconsquartic}, we find $P=1$ and $\mathcal{P}=T^3-6T^2+8T-1$. Write $L/\Q$ for the splitting field of $\mathcal{P}$, and $\mathcal{A}, \mathcal{B}, \mathcal{C}$ for its roots.  The discriminant of $\mathcal{P}$ is $Q=229$, which is prime. Hence, the Galois group $\Gal(L/\Q)$ is $S_3$.  The standard model $Y_1/L$ of $Y(\underline{I})$ from
Proposition \ref{reconsquartic}  is given by
\begin{equation}\label{eq:exa}
Y_1:\; \mathcal{A}x^4+\mathcal{B}y^4+\mathcal{C}z^4+x^2y^2+y^2z^2+x^2z^2=0.
\end{equation}

It has Ciani invariants $\underline{I}(Y_1)=(I_3,I_3',I_3'',I_6)\in \Q^4.$ Since $Q\neq 0$, the curve is non-special, see Lemma \ref{lem:Q}. The discriminant of this model is $\Delta(Y_1)=2^{20}$. 

Proposition \ref{prop:smoothmodel_new} implies that for all odd primes $p\neq 229$, the curve $Y$ is defined over $\Q_p^{\nr}$ and has good reduction. Hence, $f_p(Y_1)=0$ for these primes.  For $p=229$ Proposition \ref{prop:val_equal} implies that all  models $Y_0$ of $Y(\underline{I})$ defined over $\Q_p^{\nr}$ have  conductor exponent $f_{229}(Y_0)=4$.

Recall that $229\nmid \Delta(\underline{I})=2^{20}.$ This illustrates that for primes with $p\nmid \Delta(\underline{I})$, there need not exist a model of $Y(\underline{I})$ over $\Q(\underline{I})$ with good reduction at $p$. In this example, the standard model $Y_1/L$ from \eqref{eq:exa} satisfies $\Delta(Y_1)=\Delta(\underline{{I}})$. Therefore, $Y_1$ has good reduction to any prime of $L$ above $p=229$. However, there is no model of $Y(\underline{I})$  over $\Q$ with good reduction to characteristic $p=229$. 
\end{example}

\begin{remark} The invariants in Example \ref{exa:fp_small} were chosen to ensure that  the  discriminant $\Delta(\underline{I})=2^{20}I_3(I_3'')^4I_6^2$  is a power of $2$.  In other words, the curve $Y(\underline{I})$ has potentially good reduction at all odd primes. However, we have seen that  the  relatively large prime $p=229$ divides the conductor of every $\Q$-model of $Y(\underline{I})$.
Proposition \ref{prop:val_equal} gives a geometric interpretation: the automorphism group of the reduction $\overline{Y}$ of $Y(\underline{I})$ to characteristic $p=229$ is strictly larger than the Ciani subgroup $V$. In the example, $p=229$ is the only odd prime for which this happens.
\end{remark}

\subsection{The case of potentially good hyper\-elliptic reduction} \label{sec:hyper_cond}

In this section, we prove results similar to those of the previous section, in the case of potentially good hyperelliptic reduction. As in Section \ref{sec:smoothmodelquartic}, we fix a prime $p>3$ and set $K=\Q_p^{\nr}$.  Let $\underline{I}=(I_3, I'_3, I''_3, I_6)\in K^4$ be a set of Ciani invariants.  We assume that $Y(\underline{I})$ has potentially good hyperelliptic reduction. We can then assume the invariants to be  normalised because of Remark \ref{normI6}.   Hence, Lemma \ref{lem:potgoodhyper} implies that there exists an integer $e>0$ such that  the Ciani invariants $\underline{I}$ satisfy: 
\begin{equation}\label{eq:valinvariants_hyper}\nu(I_3)=0,\quad
\nu(I_3')\geq e,\quad \nu(I_3'')=2e, \text{  and } \nu(I_6)=3e.  
\end{equation}

Namely, $e=\nu(I_6)/\nu(I''_3).$
We note that $\nu(P)=0$ and that the discriminant $Q$ of the polynomial $\mathcal{P}$ defined in Section \ref{sec:recons} satisfies $\nu(Q)\geq 6e$.
Let $L/K$ be the splitting field of the polynomial $\mathcal{P}$. As in Section \ref{sec:recons}, we write $\cA, \cB, \cC$ for the roots of $\mathcal{P}$.

The fact that $\nu(P)=0$ implies that we are in the situation of Proposition \ref{reconsquartic}(a), so there exists a a standard $L$-model 
\begin{equation}Y_1:\; Ax^4+By^4+Cz^4+ay^2z^2+bx^2z^2+cx^2y^2=0\end{equation} 
with normalised invariants and $\underline{I}(Y_1)=\underline{I}$. Remark \ref{rmk:roots}(a) implies that  $\mathcal{A}=Aa^2$, $\mathcal{B}=Bb^2$ and $\mathcal{C}=Cc^2$.

Define, with the notation in  Subsection \ref{sec:recons},
\begin{multline}\label{def:invR}R=(\mathcal{A}-\mathcal{B})(\mathcal{A}-\mathcal{C})+(\mathcal{B}-\mathcal{A})(\mathcal{B}-\mathcal{C})+(\mathcal{C}-\mathcal{A})(\mathcal{C}-\mathcal{B})=\\ \mathcal{A}^2+\mathcal{B}^2+\mathcal{C}^2-\mathcal{B}\mathcal{C}-\mathcal{C}\mathcal{A}-\mathcal{A}\mathcal{B}=S_1^2-3S_2.\end{multline}

The following lemma allows us to determine the degree $[L:K]$ in terms of valuations of some invariants without needing to compute $L$. This is useful for the case distinction in Theorem \ref{thm:main}.

\begin{lemma}\label{lem:sethyper} In the setting above, let $\underline{I}\in K^4$ be a normalised set of invariants such that $Y(\underline{I})$  has potentially good hyperelliptic reduction to characteristic $p>3$. 
\begin{itemize}
\item[(a)] We have that $\nu(\Delta_a)=\nu(\Delta_b)=\nu(\Delta_c)=e$.
\item[(b)] If $[L:K]=1$ then   $\nu(Q)\equiv 0,2,4\pmod{6}$.  Moreover, if $\nu(Q)\not\equiv0$ then $\nu(Q)>3\nu(R)$.
\item[(c)] If $[L:K]=2$ then   $\nu(Q)\equiv 1,3,5\pmod{6}$. 
\item[(d)] If $[L:K]=3$ then   $\nu(Q)\equiv 2,4\pmod{6}$ and $\nu(Q)\leq3\nu(R)$. 
\end{itemize}
\end{lemma}

\begin{proof} In order to prove (a), notice that $\nu(I_6)=\nu(\Delta_a\Delta_b\Delta_c)=3e$ but $\nu(I'_3)=\nu(A\Delta_a+B\Delta_b+C\Delta_c)\geq e$ with $\nu(I_3)=\nu(ABC)=0$, so the minimum of the valuations of the $\Delta_i$, if attained only once, cannot be smaller than $e$ and $\nu(I)=\nu(AB\Delta_a\Delta_b+BC\Delta_b\Delta_c+CA\Delta_c\Delta_a)\geq2e$ (see \ref{eq:Irel}), so if the minimum is smaller than $e$ it cannot be attained twice.

In order to prove the remaining statements, notice that $\nu(I_3)=0$ implies that we can assume $A=B=C=1$ and hence $a=2+\pi^ea'$, $b=2+\pi^eb'$ and $c=2+\pi^ec'$ with $\nu(a'b'c')=0$. So, $\nu(Q)=6e+2(\nu(a'-b')+\nu(b'-c')+\nu(c'-a'))$. If $L=K$ the last three valuations are integers, if $[L:K]=2$, then two of them are equal and belonging to $\mathbb{Z}_{\geq0}/2$ and the third one is in $1/2+\mathbb{Z}_{\geq0}$. If $[L:K]=3$, the three of them are equal and belong to $1/3+\mathbb{Z}_{\geq0}$ or $2/3+\mathbb{Z}_{\geq0}$. The claims about the valuation of $Q$ modulo $6$ follow.

Assume in the case $[L:K]=1$ and $\nu(Q)\not\equiv 0\pmod{6}$ that $b'$ and $c'$ are $p$-adically closest between them than to $a'$. Then $\nu(Q)=6e+6\nu(b'-a')+2(\nu(b'-c')-\nu(b'-a'))$ and $\nu(R)=2e+2\nu(a'-b')$. In the case $[L:K]=3$, $\nu(Q)=6e+6\nu(a'-b')$ and $\nu(R)\geq 2e+2\nu(a'-b')$. So we can distinguish both cases by looking at the conditions $\nu(Q)>3\nu(R)$ or $\nu(Q)\leq3\nu(R)$.
\end{proof}

\begin{proposition}\label{prop:goodhyp_new}
Let $\underline{I}=(I_3, I'_3, I''_3, I_6)\in K^4$  be a set of normalised Ciani invariants. Assume that $Y(\underline{I})$ has potentially good hyperelliptic reduction. Let $L/K$ be the splitting field of the polynomial $\mathcal{P}$ defined in Section \ref{sec:recons}. 
\begin{itemize}
    \item[(I)] 
Assume $[L:K]=1$ or $3$.
\begin{itemize}
    \item[(a)] If $e$ is even, $Y_1(\underline{I})$ has good hyperelliptic reduction over $L$. 
\item[(b)] If $e$ is odd, $Y_1(\underline{I})$ acquires good reduction over a  quadratic extension $L'/L$. 
\end{itemize}
\item[(II)] If $[L:K]=2$ then $Y_1(\underline{I})$ has good reduction over $L$.
\end{itemize}
\end{proposition}

\begin{proof}
Because of Lemma \ref{lem:sethyper}(a), and as we did in its proof, we can assume $A=B=C=1$ and work with the model:
\begin{equation}\label{eq:smoothhyper0}
Y_1:\; (x^2+y^2+z^2)^2+ \pi^e(a'y^2z^2+b'z^2x^2+c'x^2y^2)=0.
\end{equation}

Set $H(x,y,z)=x^2+y^2+z^2$ and $G(x,y,z)=a' y^2z^2+b' z^2x^2+ c' x^2y^2$. We are in the setting of \cite[Theorem 1.4]{LLLR}.
Let $\pi^e\in K$ be an element of valuation $e$. By \cite[Proposition 1.2]{LLLR}, a model  $\mathcal{Y}$ defined over the ring of integers of the extension $M=L(\pi^{e/2})$ of $L$  is given by
\begin{align}\label{eq:smoothhyper_new}
\left\lbrace
\begin{array}{c}
 t^2+G=0,\\
\pi^{e/2} t-H=0.
\end{array}
\right.
\end{align}

From \cite[Lemma 14]{BCKKLS}  it follows that the special fiber of $\mathcal{Y}$ is smooth.  

In the case that $e$ is even, we have that $M=L$, and it follows that $Y_1$ has already good reduction over $L$. In the case that $[L:K]=2,$ we have that $\pi^{e/2}\in L$ and hence that $M=L$, regardless of whether $e$ is odd or even.  
\end{proof}

\begin{proposition}\label{prop:hyperconductor}
Let $\underline{I}$ be a normalised set of invariants with $\Delta(\underline{I})\neq 0$, and assume that  $Y(\underline{I})$ has  good hyperelliptic  reduction. Let $e$ be as introduced in Equation (\ref{eq:valinvariants_hyper}).
\begin{itemize}
    \item[(a)] There exists a $K$-model $Y_0$ of $Y(\underline{I})$ with $f_p(Y_0)$ equal to: 
\end{itemize}

    \begin{table}[h!] \centering
\begin{tabular}{c|c|c|c|c|}
\multicolumn{2}{c|}{} & \multicolumn{3}{c|}{$[L:K]$} 
\\ \cline{3-5}
\multicolumn{2}{c|}{}  & 1 & 2 & 3
\\ \hline
 \multirow{2}{*}{$e$} & odd & 6 & 2 & 6
\\ \cline{2-5}
 & even & 0 & 4 & 4
 \\ \hline
\end{tabular}
\end{table}

\begin{itemize}
    \item[(b)] If the curve is non-special, then any $K$-model $Y_0$ of $Y(\underline{I})$ has the same conductor exponent.
\end{itemize}

\end{proposition}

\begin{proof} The proof goes in the same direction as the proof of Proposition \ref{prop:val_equal}. We start by computing an isomorphism between a $K$-model and a model with good reduction over a field extension. This allows to compute the action of the inertia and we can then apply Proposition \ref{prop:fp} to compute the conductor. 

As in the proof of Proposition \ref{prop:goodhyp_new} we can assume $A=B=C=1$ and we work with the model 
\begin{align*}
\mathcal{Y}:\,\left\lbrace
\begin{array}{c}
 t^2+a' y^2z^2+b' z^2x^2+ c' x^2y^2=0\\
\pi^{e/2} t-(x^2+y^2+z^2)=0,
\end{array}
\right.
\end{align*}
defined over  $M=L(\pi^{e/2})$  and that has good hyperelliptic reduction. 

The isomorphism $\phi$, defined in Proposition \ref{prop:reconsK}(a) for each case of $[L:K]$, applies to the variables $x,y,z$ of $\mathcal{Y}$ and it provides a model $\mathcal{Y}_0$ defined over $K$ with variables $(x_0,y_0,z_0)=\phi^{-1}(x,y,z)$  and $t_0=\pi^{e/2}t$. 

For each case depending on the degree $[L:K]$ and the parity of $e$, we compute the action of the inertia on $\mathcal{Y}$, and we use Proposition \ref{prop:fp} to compute the conductor exponent, as follows.

For $[L:K]=1$, the inertia is generated by $(\overline{x},\overline{y},\overline{z},\overline{t})\mapsto(\overline{x},\overline{y},\overline{z},\overline{t})$ or $(\overline{x},\overline{y},\overline{z},\overline{t})\mapsto(\overline{x},\overline{y},\overline{z},-\overline{t})$ depending on $e$ being even or odd, hence the conductor is $2\cdot3-2\cdot3=0$ or $2\cdot3-2\cdot0=6$.

For $[L:K]=2$, assuming $a'$ is the one element in $K$, the inertia is generated by $\sigma: (\overline{x},\overline{y},\overline{z},\overline{t})\mapsto(\overline{x},\overline{z},\overline{y},\pm\overline{t})$, and the conductor is then $2\cdot3-2\cdot1=4$ or $2\cdot3-2\cdot2=2$. We notice that the automorphism $\sigma$ has no fixed points if and only if $\sigma^\ast t=-t$, and this happens if and only if $e$ is odd. If $e$ is even, the automorphism $\sigma$ of $\overline{Y}$ has $4$ fixed points. This can be checked using the classification of automorphism groups of hyperelliptic curves in \cite[Section 3.1]{LR11}.

Finally, for $[L:K]=3$, the inertia is generated by $(\overline{x},\overline{y},\overline{z},\overline{t})\mapsto(\overline{y},\overline{z},\overline{x},\pm\overline{t})$ and the conductor is then $2\cdot3-2\cdot1=4$ or $2\cdot3-2\cdot0=6$, depending on whether $e$ is even or odd.

We repeat the computations for each of the twists of $\mathcal{Y}_0$ described in Lemma \ref{lem:twists}.
\end{proof}

\begin{remark}\label{rem:autospecial} Again, as in the case of potentially good quartic reduction (Remark \ref{rem:aut}), when $[L:K]>1$, we have that the special fiber, this time a hyperelliptic curve, has automorphisms that are not contained in the group generated by the fixed Ciani subgroup and the hyperelliptic involution.
\end{remark}

\section{Comparison between the conductor and the discriminant}\label{Sec:bounds}

Let  $(K, \nu)$ be a complete local field of characteristic zero with valuation $\nu$, whose residue field is an algebraically closed field $k$ of odd characteristic $p>0$. If $E/K$ is an elliptic curve, Ogg's formula \cite{Ogg} implies that
\begin{equation}\label{eq:ineq}
f_p(E)\leq \nu(\Delta(E)),
\end{equation}
where $\Delta(E)$ is the discriminant. It is natural to ask whether this inequality also holds for plane curves of arbitrary degree.  If the inequality  holds, a list of all curves with bounded discriminant as in \cite{Sutherland} also contains all curves whose conductor is bounded (by a certain slightly different bound). This inequality is discussed for example in \cite[Section 5]{BKSW} for Picard curves and in \cite{OS} for hyperelliptic curves.

Our results prove that \eqref{eq:ineq} hold in a rather special case. The following is a corollary of Theorem \ref{thm:main}.

\begin{corollary}
Let $K=\Q_p^{\nr}$ with $p>3$ and that $Y_0/K$ is a non-special Ciani curve with potentially good reduction at $p$.
Then 
\[
f_p(Y_0)\leq \nu(\Delta(Y_0)).
\]
\end{corollary}

 We note that every curve $Y_0$ occurs in Lemma \ref{lem:twists} as one of the $K$-twist for some tuple $[\underline{I}]\in \mathbb{P}^3_{1,1,1,2}(K)$ of Ciani invariants.  Moreover, since we assume that $Y_0$ is non-special, the conductor exponent $f_p(Y_0)$ only depends on the $\overline{K}$-isomorphism class and the field $K$. Replacing $K$ by a larger field decreases the conductor exponent in general. For this reason, we mostly  assumed that $K=\Q_p^{\nr}$ in this paper.

\begin{proof}
 Let $\underline{I}\in K^4$ be a normalised set of Ciani invariant such that $Y_0\otimes_K \overline{K}$ is isomorphic to $Y[\underline{I}]$. Since we assume that $Y_0$ has potentially good reduction, we can normalise the invariants over $K$, see Remark \ref{normI6}.  
 
Let $M/K$ be the extension over which $Y_0$ acquires good reduction, and let $Y_2(\underline{I})/M$ be the model with good reduction.  The extension $M/K$ is at most tamely ramified, Corollary \ref{cor:boundLK}. Therefore, Proposition \ref{prop:fp} implies that $f_p(Y_0)\leq 6.$

 We only consider the case that $Y_0$ has potentially good quartic reduction, but not good reduction over $K$.  Moreover, we assume that we are in case (a) of Proposition \ref{reconsquartic}. The case (b) and the case of potentially good hyperelliptic reduction are very similar.  Since $Y_2/M$ has good quartic reduction, the discriminant $\Delta(Y_2(\underline{I}))$ has valuation $0$, see \eqref{eq:Delta2}. The valuation of the determinant of the isomorphism  $\psi\circ \phi:Y_0(\underline{I})\otimes_K M\to Y_2(\underline{I})$ is positive, since we assume that $M\neq K$. Moreover, $[M:K]\nu(\det(\psi\circ\phi))$ is an integer. Using Lemma \ref{trans}, we conclude that
 \[
 \nu(\Delta(Y_0(\underline{I}))\geq \frac{36}{[M:K]}\geq 6. 
 \]
The statement of the corollary in this case follows.
\end{proof}

In \cite{BK94}, we find an upper bound on the conductor exponent only in terms of the genus and the ramification index. In our more special setup, we improve on the upper bound of Brumer--Kramer only for $p=5,7$. The reason is that a Ciani curve acquires stable reduction over a tame extension if $p>3$. For arbitrary curves of genus $3$ this only holds if $p>7.$

\bibliographystyle{alphaabbr}
\bibliography{synthbib}

\end{document}